\documentclass[11pt]{amsart}
\usepackage{epstopdf}
\usepackage[numbers,sort&compress]{natbib}
\bibpunct[, ]{[}{]}{,}{n}{,}{,}
\makeatletter
\def\NAT@def@citea{\def\@citea{\NAT@separator}}
\makeatother
\usepackage{xcolor}

\usepackage{amsmath}
\usepackage{amssymb}
\usepackage{amsthm}
\usepackage{amsfonts}
\usepackage{arydshln}
\usepackage{enumerate}
\usepackage[thinlines]{easybmat}
\numberwithin{equation}{section}
\newcommand{\hsp}{\mathcal{H}}
\newcommand{\oL }{\mathcal{L}(\hsp) }

\newcommand{\dfc}[1]{D_{\Theta_#1}}

\newcommand{\dcf}[1]{\widetilde{D}_{\Theta_#1}}
\newcommand{\dfct}[1]{{D}_{\widetilde{\Theta}_#1}}

\newcommand{\B}{{\mathbf{B}}}
\newcommand{\h}{\mathcal{H}}

\newcommand{\hh}{{\mathbf{H}}}
\newcommand{\ke}{{\mathbf{k}}}
\newcommand{\C}{{\mathbf{C}}}
\newcommand{\F}{{\mathbf{F}}}
\newcommand{\J}{{\mathbf{J}}}
\newcommand{\T}{{\mathbf{T}}}

\newcommand{\z}{{z}}
\newtheorem{theorem}{Theorem}[section]
\newtheorem{proposition}[theorem]{Proposition}
\newtheorem{corollary}[theorem]{Corollary}
\newtheorem{lemma}[theorem]{Lemma}
\theoremstyle{definition}

\newtheorem{remark}[theorem]{Remark}
\numberwithin{equation}{section}

\title[Characterizations of MATTO's]{Characterizations of Matrix valued asymmetric truncated Toeplitz operators}

\author[R. Khan]{Rewayat Khan}
\address{R. Khan, Abbottabad University of Science and Technology, Pakistan}
\email{rewayat.khan@sabanciuniv.edu}

\author[Y. Ameur]{Yagoub Ameur}
\address{Y. Ameur, Laboratoire de mathematiques pures et appliques Universite de Amar telidji Laghouat
Algerie 03000}
\email{a.yagoub@lagh-univ.dz}

\author[J. Khan]{Jamroz Khan}
\address{J. Khan, Government College of Management Science II Peshawar, Pakistan}
\email{jamroz.khan73@gmail.com}

\keywords{model spaces, model for a contraction, truncated Toeplitz operator, asymmetric truncated Toeplitz operator, matrix valued truncated Toeplitz operator, matrix valued asymmetric truncated Toeplitz operator.}
\subjclass[2010]{Primary 47B35, Secondary 47B32, 30D20}
\begin{document}
\begin{abstract}{We characterize matrix-valued asymmetric truncated Toeplitz operators (which are compressions of
multiplication operators acting between two possibly different model spaces) by using compressed shifts, modified compressed shifts and shift invariance.
	}
\end{abstract}
\maketitle

\section{Introduction}

Let $H^2$ be the classical Hardy space in the unit disk $\mathbb{D}=\{\lambda\in\mathbb{C}:|\lambda|<1\}$. Truncated Toeplitz operators (TTO's) and asymmetric truncated Toeplitz operators (ATTO's) are compressions of multiplication operator to the backward shift invariant subspaces of $H^2$ (with two possibly different underlying subspaces in the asymmetric case). Each of these subspaces is of the form $K_{\theta}=(\theta H^{2})^{\perp}=H^2\ominus \theta H^2$, where $\theta$ a complex-valued inner function: $\theta\in H^\infty$ and $|\theta(\z)|=1$ a.e. on the unit circle $\mathbb{T}=\partial\mathbb{D}=\{\z\in\mathbb{C}: |\z|=1\}$.

It is natural to consider TTO's and ATTO's defined on subspaces of vector valued Hardy space $H^2(\hsp)$ with $\hsp$ a separable complex Hilbert space. 
 A vector valued model space $K_{\Theta}\subset H^{2}(\hsp)$ is the orthogonal complement of $\Theta H^{2}(\hsp)$, that is, $K_{\Theta}= H^{2}(\hsp)\ominus  \Theta H^{2}(\hsp)$. Here $\Theta$ is an operator valued inner function: a function with values in $\oL $ (the algebra of all bounded linear operators on $\hsp$), analytic in $\mathbb{D}$, bounded and such that the boundary values $\Theta(\z)$ are unitary operators a.e. on $\mathbb{T}$. These spaces appear in connection with model theory of Hilbert space contractions (see \cite{NF}). Let $P_{\Theta}$ be the orthogonal projection from $L^2(\hsp)$ onto $K_{\Theta}$.

For two inner functions $\Theta_1,\Theta_2\in H^{\infty}(\oL)$ and $\Phi\in L^2(\oL )$ let
\begin{equation}\label{11}
A_{\Phi}^{\Theta_{1}, \Theta_{2}}f=P_{\Theta_{2}}(\Phi f),\quad f\in K_{\Theta_{1}}\cap H^{\infty}(\hsp ).
\end{equation}

 The operator $A_{\Phi}^{\Theta_{1}, \Theta_{2}}$ is called a matrix valued asymmetric truncated Toeplitz operator (MATTO), while $A_{\Phi}^{\Theta_{1}}=A_{\Phi}^{\Theta_{1}, \Theta_{1}}$ is called a matrix valued truncated Toeplitz operator (see \cite{KT}). Both are densely defined. Let $\mathcal{MT}(\Theta_{1},\Theta_{2})$ be the set of all MATTO's of the form \eqref{11} which can be extended boundedly to the whole space $K_{\Theta_{1}}$ and for $\Theta_{1}=\Theta_{2}=\Theta$ let $\mathcal{MT}(\Theta )=\mathcal{MT}(\Theta,\Theta)$.

 Two important examples of operators from $\mathcal{MT}(\Theta )$ are the model operators
\begin{equation}\label{modl}
S_{\Theta}=A_{z }^{\Theta }=A_{zI_{\hsp}}^{\Theta }\quad\text{and}\quad S_{\Theta }^*=A_{\bar z }^{\Theta }=A_{\bar zI_{\hsp}}^{\Theta  }.
\end{equation} It is known that each $C_0$ contraction with finite defect indices is unitarily equivalent to $S_{\Theta}$ for some operator valued inner function $\Theta$ (see \cite[Chapter IV]{NF}).

 Sections 2 and 3 contain preliminary material on spaces of vector valued functions (Section 2), model spaces and MATTO's (Section 3). In Section 4 we consider some model space operators and their action on $\mathcal{MT}(\Theta_{1},\Theta_{2})$. Section 5 is devoted to characterizations of MATTO's in terms of $S_{\Theta_{1}}$, $S_{\Theta_{2}}$ and their adjoints. In Section 6 we consider the notion of shift invariance of operators from $\mathcal{MT}(\Theta_{1},\Theta_{2})$. In section 7 we use modified compressed shift to characterize MATTO's.

\section{Spaces of vector valued functions and their operators}\label{s2}

Let $\hsp$ be a complex separable Hilbert space. The space $L^2(\hsp)$ consists of elements $f:\mathbb{T}\to \hsp$ of the form
\begin{equation}\label{1}
 \begin{array}{rl}
 f(\z)=\sum\limits_{n=-\infty}^{\infty}a_{n}\z^n&  (\text{a.e. on }\mathbb{T})\\
  \text{with}&\{a_{n}\}\subset\hsp\text{ such that}  \sum\limits_{n=-\infty}^{\infty}\|a_{n}\|_\hsp^{2}<\infty.\end{array}\end{equation}
It is a (separable) Hilbert space with the inner product given by
\begin{equation*}
\langle f,g\rangle_{L^{2}(\hsp)}=\int_{\mathbb{T}}\langle f(\z),g(\z)\rangle_{\hsp}\,dm(\z),\quad f,g\in L^2(\hsp).
\end{equation*}
If $f\in L^2(\hsp)$ is given by \eqref{1}, then its Fourier series converges in the  $L^2(\h)$ norm and
$$\|f\|_{L^2(\hsp)}^2=\int_{\mathbb{T}} \| f(\z)\|_\hsp^2\,dm(\z)=\sum\limits_{n=-\infty}^{\infty}\|a_{n}\|_\hsp^{2}.$$

The vector valued Hardy space $H^2(\hsp)$ is defined as the set of all the elements of $L^2(\hsp)$ whose Fourier coefficients with negative indices vanish.
Each $f\in H^2(\hsp)$, $\displaystyle f(\z)=\sum_{n=0}^{\infty} a_n \z^n$, can also be identified with a function
$$f(\lambda)=\sum\limits_{n=0}^\infty a_n\lambda^n,\quad \lambda\in\mathbb{D},$$
analytic in the unit disk $\mathbb{D}$. Denote by $P_+$ the orthogonal projection $P_+:L^2(\hsp)\to H^2(\hsp)$,
 $$P_+\left(\sum_{n=-\infty}^{\infty} a_n \z^n\right)=\sum_{n=0}^{\infty} a_n \z^n.$$

 The space of essentially bounded functions in $L^{2}(\hsp)$ is denoted by $L^{\infty}(\hsp)$ and
 $H^{\infty}(\hsp)=L^{\infty}(\hsp)\cap H^{2}(\hsp)$.

Now let $\oL $ be the algebra of all bounded linear operators on $\hsp$ equipped with the operator norm $\|\cdot\|_{\oL}$. We can define $\oL$-valued, i.e., operator valued functions. We denote these spaces by $L^{2}(\oL)$ and $H^{2}(\oL)$, respectively. The space of operator valued, essentially bounded functions on $\mathbb{T}$ is denoted by $L^{\infty}(\oL)$, and the space of bounded analytic functions in $H^{2}(\oL)$ is denoted by $H^{\infty}(\oL)$.

Note that for each $\lambda\in \mathbb{D}$ the function $\ke_\lambda(z)= (1-\bar{\lambda}z)^{-1}I_{\hsp}$ belongs to $H^{\infty}(\oL )$ and has the following reproducing property
$$\langle f, \ke_\lambda x\rangle_{L^2(\hsp)}=\langle f(\lambda),x\rangle_{\hsp}, \quad f\in H^2(\hsp).$$

To each $\F\in L^\infty(\oL)$ there corresponds a multiplication operator $M_\F:L^2(\hsp)\to L^2(\hsp)$: for $f\in L^2(\hsp)$,
$$(M_\F f)(\z)=\F(\z)f(\z)\quad \text{a.e. on }\mathbb{T}.$$
By $T_{\F}$ we will denote the compression of $M_{\F}$ to the Hardy space: $T_{\F}:H^2(\hsp)\to H^2(\hsp)$, $$T_{\F}f=P_+M_{\F}f\quad\text{for }f\in H^2(\hsp).$$

In particular, for $M_z=M_{zI_\hsp}$ we have $M_z^*=M_{\bar{z}}=M_{\bar{z}I_\hsp}$. The operator $S=T_z=M_{z|H^2(\hsp)}$ is called the (forward) shift operator. Its adjoint, the backward shift operator $S^*=T_{\bar z}$, is given by the formula
$$S^*f(z)=\bar z\big(f(z)-f(0)\big).$$

Here we assume that $\dim\hsp<\infty$ so we can consider $\oL $ as a Hilbert space with the Hilbert--Schmidt norm and we may also define as above the spaces $L^2(\oL)$ and $H^2(\oL)$.
We can decompose $L^{2}(\hsp)$ as
$L^2(\oL)=\left[zH^2(\oL)\right]^*\oplus H^2(\oL).$

For $\F\in L^2(\oL )$ the operators $M_{\F}$ and $T_{\F}$ can be  densely defined, on $L^{2}(\hsp )$ and $H^{2}(\hsp )$, respectively.
 For more details on spaces of vector valued functions we refer the reader to \cite{berc,NF}.

\section{Model spaces and MATTO's}\label{s3}

An inner function is called pure if $\|\Theta(0)\|_{\oL}<1$. Throughout this paper we consider only pure inner functions.
The model space
$$K_{\Theta}=H^{2}(\hsp)\ominus \Theta H^{2}(\hsp)$$
corresponding to an inner function $\Theta$ is invariant under the backward shift $S^*$. Moreover, by the vector valued version of Beurling's invariant subspace theorem, each closed (nontrivial) $S^*$--invariant subspace of $H^2(\hsp)$ is a model space (\cite[Chapter 5, Theorem 1.10]{berc}). Let $P_{\Theta}$ be the orthogonal projection from $L^2(\hsp)$ onto $K_{\Theta}$. Then
$$P_{\Theta}=P_+-M_{\Theta}P_+M_{\Theta^*}.$$
Note that $M_{\Theta}$ is the multiplication operator on $L^{2}(\hsp)$.

For each $\lambda\in\mathbb{D}$ we can consider
$$\ke_{\lambda}^{\Theta}(z)=\tfrac1{1-\bar\lambda z}(I_\hsp-\Theta(z)\Theta(\lambda)^*)\in H^{\infty}(\oL).$$
For each $x\in\hsp$ and $\lambda\in\mathbb{D}$, the function $\ke_{\lambda}^{\Theta}x=P_{\Theta}(\ke_{\lambda}x)$  belongs to $K_{\Theta}^{\infty}=K_{\Theta}\cap H^{\infty}(\hsp )$ and has the following reproducing property
$$\langle f, \ke_{\lambda}^{\Theta}x\rangle_{L^2(\hsp)}=\langle f(\lambda), x\rangle_\hsp\quad\text{for every }f\in K_{\Theta}.$$
It follows in particular that $K_{\Theta}^{\infty}=K_{\Theta}\cap H^{\infty}(\hsp)$ is a dense subset of $K_{\Theta}$.

Now let $\Theta_{1},\Theta_{2}\in H^{\infty}(\oL)$ be two inner functions. For any $\Phi\in L^{2}(\oL)$ define
$$A_{\Phi}^{\Theta_{1}, \Theta_{2}}f=P_{\Theta_{2}}M_\Phi f=P_{\Theta_{2}}(\Phi f),\quad f\in K_{\Theta_{1}}^{\infty}.$$

The operator $A_{\Phi}^{\Theta_{1}, \Theta_{2}}$ is called a matrix valued asymmetric truncated Toeplitz operator (MATTO) with symbol $\Phi\in L^{2}(\oL)$. It is densely defined and if bounded, it can be extended to a bounded linear operator $A_{\Phi}^{\Theta_{1}, \Theta_{2}}:K_{\Theta_{1}}\to K_{\Theta_{2}}$ (in which case we simply say that $A_{\Phi}^{\Theta_{1}, \Theta_{2}}$ is bounded). Let us denote
$$\mathcal{MT}(\Theta_{1}, \Theta_{2})=\{A_{\Phi}^{\Theta_{1}, \Theta_{2}}:\, \Phi\in L^2(\oL)\text{ and }A_{\Phi}^{\Theta_{1}, \Theta_{2}} \text{ is bounded}\}.$$
For $\Theta_{1}=\Theta_{2}=\Theta$ we put $A_{\Phi}^{\Theta}=A_{\Phi}^{\Theta, \Theta}$ (a matrix valued truncated Toeplitz operator, MTTO) and $\mathcal{MT}(\Theta)=\mathcal{MT}(\Theta, \Theta)$.

Let $$\mathcal{D}_{\Theta}=\{(I_{\hsp}-\Theta\Theta(0)^*)x:\ x\in\hsp\}=\{\ke_{0}^{\Theta}x:\ x\in\hsp\}\subset K_{\Theta}.$$
Then for $f\in K_{\Theta}$ we have $f\perp \mathcal{D}_{\Theta}$ if and only if $f(0)=0$.
It follows that
$$(S_{\Theta}^*f)(z)=\left\{\begin{array}{cl}
\bar z f(z)&\text{for }f\perp\mathcal{D}_{\Theta},\\
-\bar z\big(\Theta(z)-\Theta(0)\big)\Theta(0)^*x&\text{for }f=\ke_{0}^{\Theta}x\in\mathcal{D}_{\Theta}.
\end{array}\right.$$
Denote (the defect operator) by $D_{\Theta}=I_{K_{\Theta}}-S_{\Theta}S_{\Theta}^*.$
Since for each $f\in H^2(\hsp)$ we have $(I_{H^2(\hsp)}-SS^*)f=f(0)$ (a constant function in $H^2(\hsp)$), it follows that for $f\in K_{\Theta}$,
\begin{equation}\label{ect}
\begin{split}
D_{\Theta}f&=(I_{K_{\Theta}}-S_{\Theta}S_{\Theta}^*)f=P_{\Theta}(I_{H^2(\hsp)}-SS^*)f\\
&=(I_{\hsp}-\Theta\Theta(0)^*)f(0)=\ke_{0}^{\Theta}f(0)\in \mathcal{D}_{\Theta}.
\end{split}
\end{equation}
More precisely,
$$D_{\Theta}f=\left\{\begin{array}{cl}
0&\text{for }f\perp\mathcal{D}_{\Theta},\\
\ke_{0}^{\Theta}(I_{\hsp}-\Theta(0)\Theta(0)^*)x&\text{for }f=\ke_{0}^{\Theta}x\in\mathcal{D}_{\Theta}.
\end{array}\right.$$

Since $\ke_{0}^{\Theta}$ is invertible in $H^{\infty}(\oL)$, the formula
\begin{equation*}
\Omega_{\Theta}(\ke_{0}^{\Theta}x)=x,\quad x\in\hsp,
\end{equation*}
gives a well defined operator $\Omega_{\Theta}:\mathcal{D}_{\Theta}\to\hsp$. Clearly, $\Omega_{\Theta}$ is bounded (here for example as an operator acting between two finite dimensional Hilbert spaces). Since $\hsp$ can be identified with a subspace of $H^2(\hsp)$ (the space of all constant $\hsp$--valued functions), $\Omega_{\Theta}$ can be seen as an operator from $\mathcal{D}_{\Theta}$ into $H^2(\hsp)$. For each $f\in K_{\Theta}$ we then have
\begin{equation}
\label{wyk1}
\Omega_{\Theta}D_{\Theta}f=\Omega_{\Theta}(\ke_{0}^{\Theta}f(0))=f(0)=(I_{H^2(\hsp)}-SS^*)f.
\end{equation}

\section{MATTO's and some model space operators}\label{s4}
In \cite{RK} the author considers the generalized Crofoot transform. A bounded linear operator $W\in \oL$ is called $\text{a}$ contraction if $\|W\|_{\oL}\leq 1$ and $\text{a}$ strict contraction if $\|W\|_{\oL} < 1$. The operators $D_{W}=(I-W^{*}W)^{\frac{1}{2}}$ and $D_{W^{*}}=(I-WW^{*})^{\frac{1}{2}}$ are called the defect operators of $W$. For a pure inner function $\Theta\in H^{\infty}(\oL )$ and $W\in \oL $ such that $\|W\|_{\oL}<1$ define the generalized Crofoot transform $J_{W}^{\Theta}: L^2(\hsp)\to L^2(\hsp )$ by
$$J_{W}^{\Theta}f=D_{W^{*}}(I_{L^2(\hsp)}-\Theta W^{*})^{-1}f,\quad f\in L^2(\hsp). $$
Then $J_{W}^{\Theta}$ is unitary and maps $K_{\Theta}$ onto $K_{\Theta^{W}}$, where
$$\Theta^{W}(z)=-W+D_{W^{*}}(I_{L^2(\hsp)}-\Theta(z) W^{*})^{-1}\Theta(z) D_{W}.$$
The following theorem describes the action of the Crofoot transform on $\mathcal{MT}(\Theta_{1}, \Theta_{2})$.
\begin{theorem}{\em\cite{RK2}}
	Let $\Theta_{1},\Theta_{2}\in H^{\infty}(\oL)$ be two pure inner functions and let $W_1,W_2\in \oL$ be such that $\|W_1\|_{\oL}<1$ and $\|W_2\|_{\oL}<1$. A bounded linear operator $A:K_{\Theta_1}\to K_{\Theta_2}$ belongs to $\mathcal{MT}(\Theta_{1}, \Theta_{2})$ if and only if $J_{W_{2}}^{\Theta_2}A(J_{W_{1}}^{\Theta_1})^{*}$ belongs to $\mathcal{MT}(\Theta_{1}^{W_1}, \Theta_{2}^{W_2})$. More precisely, $A=A_{\Phi}^{\Theta_{1}, \Theta_{2}}\in \mathcal{MT}(\Theta_{1}, \Theta_{2})$ if and only if $J_{W_{2}}^{\Theta_2}A(J_{W_{1}}^{\Theta_1})^{*}=A_{\Psi}^{\Theta_{1}^{W_1}, \Theta_{2}^{W_2}}\in\mathcal{MT}(\Theta_{1}^{W_1}, \Theta_{2}^{W_2})$ with
	\begin{equation*}
	\Psi = D_{W_{2}^{*}}(I_{\oL}-\Theta_{2}W_{2}^{*})^{-1}\Phi D_{W_{1}^{*}}(I_{\oL}+\Theta^{W_1}_{1} W_{1}^{*})^{-1}.
	\end{equation*}
\end{theorem}

Recall that if $\Theta\in H^{\infty}(\oL )$ is an inner function, then so is
$\widetilde{\Theta}(z)=\Theta(\bar z)^*.$
Let us now consider the map $\tau_{\Theta}:L^2(\hsp)\to L^2(\hsp)$ defined for $f\in L^2(\hsp)$ by
\begin{equation}
\label{tau}(\tau_{\Theta}f)(z)=\bar z\Theta(\bar z)^*f(\bar z)=\bar z\widetilde{\Theta}( z)f(\bar z)\quad\text{a.e. on }\mathbb{T}.
\end{equation}
The map $\tau_{\Theta}$ is an isometry and its adjoint $\tau_{\Theta}^*=\tau_{\widetilde{\Theta}}$ is also its inverse. Hence $\tau_{\Theta}$ is unitary. Moreover, it is easy to verify that
$$\tau_{\Theta}(\Theta H^2(\hsp))\subset H^2(\hsp)^{\perp}\quad\text{and}\quad \tau_{\Theta}( H^2(\hsp)^{\perp})\subset \widetilde{\Theta} H^2(\hsp),$$
which implies that
$ \tau_{\Theta}(K_{\Theta})=K_{\widetilde{\Theta}}.$
\begin{theorem}\label{ttau}
	Let $\Theta_{1},\Theta_{2}\in H^{\infty}(\oL)$ be two pure inner functions. A bounded linear operator $A:K_{\Theta_1}\to K_{\Theta_2}$ belongs to $\mathcal{MT}(\Theta_{1}, \Theta_{2})$ if and only if $\tau_{\Theta_2}A\,\tau_{\Theta_1}^{*}$ belongs to $\mathcal{MT}(\widetilde{\Theta}_{1}, \widetilde{\Theta}_{2})$. More precisely, $A=A_{\Phi}^{\Theta_{1}, \Theta_{2}}\in \mathcal{MT}(\Theta_{1}, \Theta_{2})$ if and only if $\tau_{\Theta_2}A\,\tau_{\Theta_1}^{*}=A_{\Psi}^{\widetilde{\Theta}_{1}, \widetilde{\Theta}_{2}}\in\mathcal{MT}(\widetilde{\Theta}_{1}, \widetilde{\Theta}_{2})$ with
	\begin{equation}\label{symbda}
	\Psi(z)=\Theta_2(\bar z)^*\Phi(\bar z)\Theta_1(\bar z)=\widetilde{\Theta}_2( z)\Phi(\bar z)\widetilde{\Theta}_1( z)^*\quad\text{a.e. on }\mathbb{T}.
	\end{equation}
\end{theorem}
\begin{proof}
	Let $A:K_{\Theta_1}\to K_{\Theta_2}$ be a bounded linear operator. Assume that $A=A_{\Phi}^{\Theta_{1}, \Theta_{2}}\in \mathcal{MT}(\Theta_{1}, \Theta_{2})$ with some $\Phi\in L^2(\oL)$, and take $f\in K_{\widetilde{\Theta}_1}^{\infty}$ and $g\in K_{\widetilde{\Theta}_2}^{\infty}$. Note that $\tau_{\widetilde{\Theta}_1}f\in K_{{\Theta}_1}^{\infty}$ and $\tau_{\widetilde{\Theta}_2}g\in K_{{\Theta}_2}^{\infty}$. Therefore
	\begin{align*}
	\langle \tau_{\Theta_2}A\,\tau_{\Theta_1}^{*}f,g \rangle_{L^2(\hsp)}&=\langle A_{\Phi}^{\Theta_{1}, \Theta_{2}}\tau_{\widetilde{\Theta}_1}f,\tau_{{\Theta}_2}^*g \rangle_{L^2(\hsp)}\\&=\langle {\Phi}\,\tau_{\widetilde{\Theta}_1}f,\tau_{{\Theta}_2}^*g \rangle_{L^2(\hsp)}=\langle \tau_{{\Theta}_2}({\Phi}\,\tau_{\widetilde{\Theta}_1}f),g \rangle_{L^2(\hsp)}\\
	&=\int_{\mathbb{T}}\langle \bar z\widetilde{\Theta}_2(z)({\Phi}\,\tau_{\widetilde{\Theta}_1}f)(\bar z),g(z) \rangle_{\hsp}\,dm(z)\\
	&=\int_{\mathbb{T}}\langle \bar z\widetilde{\Theta}_2(z){\Phi}(\bar z)z{\Theta}_1(\bar z)f( z),g(z) \rangle_{\hsp}\,dm(z)\\	
	&=\int_{\mathbb{T}}\langle  \Psi( z)f( z),g(z) \rangle_{\hsp}\,dm(z)=	\langle A_{\Psi}^{\widetilde{\Theta}_{1}, \widetilde{\Theta}_{2}}f,g \rangle_{L^2(\hsp)}
	\end{align*}
	with $\Psi\in L^2(\oL)$ given by \eqref{symbda}.
	
	Now, if $\tau_{\Theta_2}A\,\tau_{\Theta_1}^{*}=A_{\Psi}^{\widetilde{\Theta}_{1}, \widetilde{\Theta}_{2}}\in\mathcal{MT}(\widetilde{\Theta}_{1}, \widetilde{\Theta}_{2})$ for some $\Psi\in L^2(\oL)$, then  $A=\tau_{\widetilde{\Theta}_2}A_{\Psi}^{\widetilde{\Theta}_{1}, \widetilde{\Theta}_{2}}\tau_{\widetilde{\Theta}_1}^{*}$ and by the first part of the proof  $A=A_{\Phi}^{\Theta_{1}, \Theta_{2}}\in \mathcal{MT}(\Theta_{1}, \Theta_{2})$ with
	\begin{equation}\label{wykwyk}
	\begin{split}
	\Phi(z) =\widetilde{\Theta}_2(\bar z)^*\Psi(\bar z)\widetilde{\Theta}_1(\bar z)={\Theta}_2( z)\Psi(\bar z){\Theta}_1( z)^*\quad\text{a.e. on }\mathbb{T}.
	\end{split}
	\end{equation}
	Hence $\Psi( z)={\Theta}_2( \bar z)^*\Phi(\bar z){\Theta}_1( \bar z)$ and \eqref{symbda} is satisfied.
\end{proof}

Denote
$\widetilde{D}_{\Theta}= I - S_{\Theta}^*S_{\Theta}.$
Applying Theorem \ref{ttau} to the model operator $S_{\Theta}$ we obtain
\begin{equation}\label{sz}
\tau_{\Theta}S_{\Theta}\tau_{\Theta}^*=\tau_{\Theta}S_{\Theta}\tau_{\widetilde{\Theta}}=S_{\widetilde{\Theta}}^*
\end{equation}
(see \cite[p. 1001]{KT}). It follows that
\begin{equation}\label{ddd}
\widetilde{D}_{\Theta}=\tau_{\widetilde{\Theta}}{D}_{\widetilde{\Theta}}\tau_{\Theta}=\tau_{\widetilde{\Theta}}{D}_{\widetilde{\Theta}}\tau_{\widetilde{\Theta}}^*
\end{equation}
and by \eqref{ect},
$$\widetilde{D}_{\Theta}f=\tau_{\widetilde{\Theta}}\big(\ke_0^{\widetilde{\Theta}}(\tau_{\Theta}f)(0)\big)\quad\text{for all }f\in K_{\Theta}.$$
For $\lambda\in\mathbb{D}$ let $\widetilde{\ke}_{\lambda}^{\Theta}x=\tau_{\widetilde{\Theta}}(\ke_{\bar\lambda}^{\widetilde{\Theta}}x)$, $x\in\hsp$. Then (a.e. on $\mathbb{T}$)
\begin{align*}
\widetilde{\ke}_{\lambda}^{\Theta}(z)x&=\tau_{\widetilde{\Theta}}(\ke_{\bar\lambda}^{\widetilde{\Theta}}(z)x)
=\tfrac{1}{z-\lambda}(\Theta(z)-{\Theta}( \lambda))x\in K_{\Theta}.
\end{align*}
In particular,
$$\widetilde{\ke}_{0}^{\Theta}(z)x=\bar z(\Theta(z)-{\Theta}(0))x$$
and
$$\widetilde{D}_{\Theta}f=\widetilde{\ke}_0^{{\Theta}}(\tau_{\Theta}f)(0)\in\widetilde{\mathcal{D}}_{\Theta},$$
where
$$\widetilde{\mathcal{D}}_{\Theta}=\{ \widetilde{\ke}_0^{{\Theta}}x: x\in \hsp\}=\{ \bar{z}(\Theta(z)-\Theta(0))x: x\in \hsp\}.$$

Observe that for $f\in K_{\Theta}$, $x\in\hsp$,
\begin{align*}
\langle f, \widetilde{\ke}_{\lambda}^{\Theta}x\rangle_{L^{2}(\hsp)}&=
\langle f,\tau_{\widetilde{\Theta}}(\ke_{\bar\lambda}^{\widetilde{\Theta}}x)\rangle_{L^{2}(\hsp)}
=\langle \tau_{\Theta}f,\ke_{\bar\lambda}^{\widetilde{\Theta}}x\rangle_{L^{2}(\hsp)}=\langle (\tau_{\Theta}f)(\bar\lambda),x\rangle_{\hsp}.
\end{align*}
It follows that for $f\in K_{\Theta}$ we have $M_zf\in K_{\Theta}$ if and only if $f\perp\widetilde{\mathcal{D}}_{\Theta}$. Indeed, $M_zf\in K_\Theta$ if and only if $\Theta P_+(\Theta^{*}M_zf))=0$. Since
$$(\Theta^{*}M_zf)(z)=\Theta(z)^*zf(z)=(\tau_{\Theta}f)(\overline{z}),$$
we have $P_+(\Theta^{*}M_zf))=(\tau_{\Theta}f)(0)$ and so $M_{z}f\in K_{\Theta}$ if and only if
$$0=\langle(\tau_{\Theta}f)(0),x\rangle=\langle f, \widetilde{\ke}_{0}^{\Theta}x\rangle_{L^{2}(\hsp)}\quad \text{for every}\quad x\in \hsp,$$
i.e, $f\perp\widetilde{\mathcal{D}}_{\Theta}$. Therefore
$$(S_{\Theta}f)(z)=\left\{\begin{array}{cl}
 z f(z)&\text{for }f\perp\widetilde{\mathcal{D}}_{\Theta},\\
-\big(I_{\hsp}-\Theta(z)\Theta(0)^*\big)\Theta(0)x&\text{for }f=\widetilde{\ke}_{0}^{\Theta}x\in\widetilde{\mathcal{D}}_{\Theta}.
\end{array}\right.$$
Hence
$$\widetilde{{D}}_{\Theta}f=\left\{\begin{array}{cl}
0&\text{for }f\perp\widetilde{\mathcal{D}}_{\Theta},\\
\widetilde{\ke}_{0}^{\Theta}(I_{\hsp}-\Theta(0)\Theta(0)^*)x&\text{for }f=\widetilde{\ke}_{0}^{\Theta}x\in\widetilde{\mathcal{D}}_{\Theta}.
\end{array}\right.$$~\label{d}

A conjugation $J$ in a Hilbert space
$\hsp$ is an antilinear map $J:\hsp\longrightarrow \hsp$ such that $J^2=I_\hsp$ and
$\langle Jf,Jg\rangle=\langle g, f\rangle \quad \text{for all}\quad f,g\in \hsp$.
Recall that a bounded linear operator $T:\hsp\longrightarrow\hsp$ is said to be $J$-symmetric ($J$ being a conjugation on $\hsp$) if $JTJ=T^*$. We say that $T$ is complex symmetric if it is $J$-symmetric with respect
to some conjugation J (see, e.g., \cite{GP} for more details on conjugations
and complex symmetric operators).

In \cite{CKLP} the authors consider certain classes of conjugations in $L^2(\hsp)$. One such conjugation is $\J^*:L^2(\hsp)\to L^2(\hsp)$ defined for a fixed conjugation $J$ in $\hsp$ by
\begin{equation}\label{J}
(\J^*f)(z)=J(f(\overline{z}))\quad \text{a.e. on}~~\mathbb{T}.
\end{equation}
It is not difficult to verify that for $f(z)=\sum_{n=-\infty}^{\infty}a_nz^n\in L^{2}(\oL)$ we have
$$(\J^*f)(z)=\sum_{n=-\infty}^{\infty}J(a_n)z^n.$$
Hence, $\J^*$ is an $M_z$-commuting conjugation, i.e, $\J^*M_z=M_z\J^*$, and $\J^*(H^{2}(\hsp))=H^{2}(\hsp)$, $\J^*P_+=P_+\J^*$ (see \cite[Section 4]{CKLP}).

For $\F\in L^\infty(\oL)$ and an arbitrary conjugation $J$ in $\hsp$ let
\begin{equation}\label{F}
\F_J(z)=J\F(z)J\quad \text{a.e on}~~~ \mathbb{T}.
\end{equation}
Then $\F_J\in L^\infty(\oL)$. As observed in \cite{CKLP}, $\F_J\in H^\infty(\oL)$ if and only if $\F\in H^\infty(\oL)$, and $\F_J$ is an inner function if and only if $\F$ is. Clearly, $(\F_J)_J=\F.$ Let us also observe that if $\F$ is $J$-symmetric, that is, $J\F(z)J=\F(z)^*$ a.e on $\mathbb{T}$ (or equivalently $\F(\lambda)$ is $J$-symmetric for $\lambda$ in $\mathbb{D}$, see \cite{CKLP}), then $\F_J=\widetilde{\F}$, where $\widetilde{\F}(z)=\F(\bar z)^{*}$.
Note that $\F_J$ is also defined for $\F\in L^{2}(\oL)$ and
\begin{equation}\label{JM}
\J^*M_\F=M_\F\J^*.
\end{equation}
\begin{proposition}\cite{CKLP}\label{a,b,c,d}
Let $\Theta\in H^\infty(\oL)$ be a pure inner function and let $J$ be a conjugation \text{on} $\hsp$. Then
\end{proposition}
\begin{enumerate}[(a)]
		\item $\J^*(\Theta H^{2}(\hsp))=\Theta_JH^2(\hsp)$;\label{a}
		\item $\J^*P_\Theta=P_{\Theta_J} \J^*$;\label{b}
\item $\J^*(K_\Theta)=K_{\Theta_J}$;\label{c}
\item $\J^*(\ke_{\lambda}^{\Theta}x)=\ke_{\overline{\lambda}}^{\Theta_{\J}}Jx$.\label{d}
\end{enumerate}

\begin{theorem}\label{4.4}
Let $\Theta_1,\Theta_2\in H^\infty(\oL)$ be two pure inner functions and let $J_1, J_2$ be two conjugations \text{on} $\hsp$. A bounded linear operator $A:K_{\Theta_1}\to K_{\Theta_2}$ belongs to $\mathcal{MT}(\Theta_1, \Theta_2)$ if and only if $\J_2^*A\J_1^*$ belongs to $\mathcal{MT}((\Theta_1)_{J_1}, (\Theta_2)_{J_2})$. More precisely, $A=A_\Phi^{\Theta_{1}, \Theta_2}\in \mathcal{MT}(\Theta_1, \Theta_2)$ if and only if
$\J_2^*A\J_1^*=A_\Psi^{(\Theta_{1})_{J_1}, (\Theta_2)_{J_2}}\in \mathcal{MT}((\Theta_1)_{J_1}, (\Theta_2)_{J_2})$ with
\begin{equation}\label{equi}
\Psi(z)=J_2\Phi(\overline{z})J_1 \quad\text{a.e. on } ~~\mathbb{T}.
\end{equation}
\end{theorem}
\begin{proof}
Assume that $A=A_\Phi^{\Theta_{1}, \Theta_2}\in\mathcal{MT}(\Theta_1, \Theta_2)$ with $\Phi\in L^2(\oL)$. Let $f\in K_{(\Theta_1)_{J_1}}^\infty$. Note that $\J_1^*f\in K_{\Theta_1}^\infty$. Therefore, by Proposition \ref{a,b,c,d}\eqref{b} and \eqref{JM},
\begin{align*}
\J_2^*A\J_1^* f
&=\J_2^*P_{\Theta_2}M_{\Phi}\J_1^*f=P_{(\Theta_2)_{J_2}}\J_2^*M_\Phi\J_1^*f\\
&=P_{(\Theta_2)_{J_2}}M_\Psi f=A_{\Psi}^{(\Theta_1)_{J_1},(\Theta_2)_{J_2}}f
\end{align*}
with $\Psi$ given by $\eqref{equi}$. Thus $\J_2^*A\J_1^*\in \mathcal{MT}((\Theta_1)_{J_1}, (\Theta_2)_{J_2})$.

 On the other hand, if
$A=\J_2^*A_\Psi^{(\Theta_{1})_{J_1}, (\Theta_2)_{J_2}}\J_1^*\in\mathcal{MT}((\Theta_1)_{J_1}, (\Theta_2)_{J_2})$  with some $\Psi\in L^2(\oL)$, then $A=\J_2^*A_{\Psi}^{(\Theta_1)_{J_1},(\Theta_2)_{J_2}}\J_1^* $ and as above, $A=A_{\Phi}^{\Theta_1,\Theta_2}$ with
$$\Phi(z)=J_2\Psi(z)J_1\quad \text{a.e. on}~~ \mathbb{T}.$$
\end{proof}

In the scalar case each model space $K_{\theta}$ is equipped with a natural conjugation $C_{\theta}$ defined in terms of boundary functions by
$(C_\theta f)(z)=\theta(z)\overline{z}\overline{f(z)}$. If $\Theta\in H^\infty(\oL)$ is an inner function and $J$ is a conjugation in $\hsp$ we can similarly define $\C_{\Theta}^J:L^2(\hsp)\to L^2(\hsp)$ by
$$(\C_{\Theta}^Jf)(z)=\Theta(z)\overline{z} J(f(z))\quad \text{a.e. on }~~\mathbb{T}.$$
It is not in general an involution. A simple computation shows that $\C_{\Theta}^J$ is a conjugation if and only if $\Theta$ is $J$-symmetric.
Furthermore,
$\C_{\Theta}^J(K_\Theta)=K_\Theta$ and
$\C_{\Theta}^J=\J^*\tau_\Theta.$

By Theorem \ref{ttau} and Theorem \ref{4.4} we get the following.

\begin{theorem}\label{4.5}
Let $\Theta_1, \Theta_2\in H^{\infty}(\oL)$ be two pure inner functions and let $J_1, J_2$ be two conjugations in $\hsp$ such that $\Theta_1$ is $J_1$-symmetric and $\Theta_2$ is $J_2$-symmetric. A bounded linear operator $A:K_{\Theta_1}\to K_{\Theta_2}$ belongs to $\mathcal{MT}(\Theta_1, \Theta_2)$ if and only if $C_{\Theta_2}^{J_2}AC_{\Theta_1}^{J_1}$ belongs to $\mathcal{MT}(\Theta_1, \Theta_2)$. More precisely, $A=A_{\Phi}^{\Theta_1, \Theta_2}\in \mathcal{MT}(\Theta_1, \Theta_2)$ if and only if $C_{\Theta_2}^{J_2}AC_{\Theta_1}^{J_1}=A_{\Psi}^{\Theta_1,\Theta_2}\in \mathcal{MT}(\Theta_1, \Theta_2)$ with
\begin{equation}
\Psi(z)=J_2\Theta_2(z)^*\Phi(z)\Theta_1(z)J_1
=\Theta_2(z)J_2\Phi(z)J_1\Theta_1(z)^*\quad \text{a.e. on } \mathbb{T}
\end{equation}
\end{theorem}
For the scalar version of Theorem \ref{4.5} see \cite{BM}.
\begin{remark}
Recall that in the scalar case $\hsp=\mathbb{C}$ every TTO on the model space $K_\theta$ is $C_\theta$-symmetric, i.e.,
$$C_\theta A_\varphi^\theta C_\theta=(A_\varphi^\theta)^*=A_{\overline{\varphi}}^\theta$$
(see, e.g., \cite{sar}). In the vector valued case, the equality
\begin{equation}\label{eq: 4.12}
C_{\Theta}^J A_{\Phi}^\theta C_{\Theta}^J=A_{\Phi^*}^\Theta.
\end{equation}
is not necessarily true for an arbitrary $\Phi\in L^2(\oL)$ (even though we assume here that $\Theta$ is $J$-symmetric). It is however satisfied if also $\Phi$ is $J$-symmetric and \text{commutes} with $\Theta$ (see \cite{KT}).
\end{remark}

\section{Characterizations with compressed shift operators}\label{s5}

In \cite{KT}(see Theorem 5.2 and Remark 5.4) characterizations of matrix valued truncated Toeplitz operators in
$\mathcal{MT}(\Theta)$ were given by using the model operators $S_{\Theta}$, $S_{\Theta}^*$ and the defect operators $D_{\Theta}$, $\widetilde{D}_{\Theta}$. These characterizations generalized D. Sarason's results \cite{sar}. Here we obtain analogous results for matrix valued asymmetric truncated Toeplitz operators from $\mathcal{MT}(\Theta_{1}, \Theta_{2})$. We use a reasoning analogous to that from \cite{KT} (see \cite{BM} for the scalar case).

\begin{lemma}\label{10}
	If $\Phi\in H^2((\mathcal{L}(\hsp))$, then
	 $$A_\Phi^{{\Theta_1},{\Theta_2}}-S_{\Theta_2}A_\Phi^{{\Theta_1},{\Theta_2}}S_{\Theta_1}^*=P_{\Theta_2}M_{\Phi}(I_{H^2(\hsp)}-SS^*)\quad\text{on }K_{\Theta_1}^{\infty}.$$
\end{lemma}
\begin{proof}
	Recall that  $S_{\Theta}=P_{\Theta}M_{z|K_{\Theta}}$ and $S_{\Theta}^*=P_+M_{\bar z|K_{\Theta}}$. Hence, for $f\in K_{\Theta_1}^{\infty}$,	
	 $$A_\Phi^{{\Theta_1},{\Theta_2}}f-S_{\Theta_2}A_\Phi^{{\Theta_1},{\Theta_2}}S_{\Theta_1}^*f=P_{\Theta_2}M_{\Phi}f-P_{\Theta_2}M_{z}P_{\Theta_2}M_{\Phi}P_{\Theta_1}M_{\bar z}f$$
	(note that $S_{\Theta_1}^*f\in K_{\Theta_1}^{\infty}$). Since $P_{\Theta_2} M_zP_{\Theta_2}=P_{\Theta_2} M_z$ on $H^2(\hsp)$, we have
	\begin{align*}
		 A_\Phi^{{\Theta_1},{\Theta_2}}f-S_{\Theta_2}A_\Phi^{{\Theta_1},{\Theta_2}}S_{\Theta_1}^*f&=P_{\Theta_2}M_{\Phi}f-P_{\Theta_2}M_{z}M_{\Phi}P_+M_{\bar z}f \\
		&=P_{\Theta_2}(M_{\Phi}-M_{z}M_{\Phi}P_+M_{\bar z})f\\
		&=P_{\Theta_2}(M_{\Phi}-M_{\Phi}M_{z}P_+M_{\bar z})f\\
		&=P_{\Theta_2}M_{\Phi}(I_{H^2(\hsp)}-SS^*)f.
	\end{align*}
\end{proof}

Recall that
$$\mathcal{D}_{\Theta}=\{ (I_{\hsp}-\Theta(z)\Theta(0)^*)x : x\in \hsp\},\quad \widetilde{\mathcal{D}}_{\Theta}=\{ \bar{z}(\Theta(z)-\Theta(0))x: x\in \hsp \},$$
while the operator $\Omega_{\Theta}:\mathcal{D}_{\Theta}\rightarrow \hsp\subset H^2(\hsp)$ is defined by  $$\Omega_{\Theta}(\ke_0^{\Theta}x)=x.$$

\begin{theorem}\label{tcara}
	Let $ {\Theta_1},{\Theta_2} \in H^{\infty}(\oL)$ be two pure inner functions and let $A:K_{\Theta_{1}}\rightarrow  K_{\Theta_{2}}$ be a bounded linear operator. Then $A$ belongs to $\mathcal{MT}(\Theta_{1}, \Theta_{2})$  if and only if there exist bounded linear operators \mbox{$B_1:\mathcal{D}_{\Theta_{1}}\to K_{\Theta_{2}}$} and $B_2: \mathcal{D}_{\Theta_{2}}\to K_{\Theta_{1}}$, such that
	\begin{equation}\label{cara}
	A-S_{\Theta_2}AS_{\Theta_1}^*=B_1\dfc{1}+\dfc{2}B_2^*.
	\end{equation}
\end{theorem}
\begin{proof} The proof follows the same line of reasoning as the proof of Theorem 5.2 in \cite{KT}.
\end{proof}

\begin{corollary}\label{cep}
	Let $ {\Theta_1},{\Theta_2} \in H^{\infty}(\oL)$ be two pure inner functions and let $A:K_{\Theta_{1}}\rightarrow  K_{\Theta_{2}}$ be a bounded linear operator.
	\begin{enumerate}[(a)]
\item If $A=A_{\Psi+\Xi^*}^{{\Theta_1},{\Theta_2}}\in \mathcal{MT}(\Theta_{1}, \Theta_{2})$, then $A$ satisfies \eqref{cara} with
\begin{equation}\label{symmm} B_1=P_{\Theta_2}M_{\Psi}\Omega_{\Theta_1}\quad\text{and}\quad B_2=P_{\Theta_1}M_{\Xi}\Omega_{\Theta_2}.
\end{equation}
		\item If $A$ satisfies \eqref{cara}, then $A=A_{\Psi+\Xi^*}^{{\Theta_1},{\Theta_2}}\in \mathcal{MT}(\Theta_{1}, \Theta_{2})$ with
			\begin{equation}\label{symm} \Psi(z)x=\big(B_1\ke_{0}^{\Theta_{1}}x\big)(z)\quad\text{and}\quad\Xi(z)x=\big(B_2\ke_{0}^{\Theta_{2}}x\big)(z),\quad x\in \hsp.
			\end{equation}
\end{enumerate}
\end{corollary}

\begin{remark}\label{rmk}
\begin{enumerate}[(a)]
\item For an inner function $\Theta\in H^{\infty}(\oL)$ denote
$$\mathcal{M}_\Theta=H^2(\oL)\ominus \Theta H^2(\oL).$$
Therefore, if a bounded linear operator $A:K_{\Theta_1}\to K_{\Theta_2}$ satisfies \eqref{cara}, then $A=A_{\Psi+\Xi^*}^{\Theta_1, \Theta_2}\in \mathcal{MT}(\Theta_1, \Theta_2)$ with $\Psi\in \mathcal{M}_{\Theta_2}$ and $\Xi\in \mathcal{M}_{\Theta_1}$ given by \eqref{symm}.
\item Recall that $A_{\Phi}^{\Theta_1, \Theta_2}=0$ if and only if
$$\Phi\in \Theta_2 H^2(\oL)+(\Theta_1 H^2(\oL))^*$$
(see \cite{RK2}).
\end{enumerate}
\end{remark}
As in \cite{KT} we can use the unitary operator $\tau_\Theta$ defined by \eqref{tau} and obtain the following theorem.

\begin{theorem}\label{tca2}
	Let $ {\Theta_1},{\Theta_2} \in H^{\infty}(\oL)$ be two pure inner functions and let $A:K_{\Theta_{1}}\rightarrow  K_{\Theta_{2}}$ be a bounded linear operator. Then $A$ belongs to $\mathcal{MT}(\Theta_{1}, \Theta_{2})$  if and only if there exist bounded linear operators \mbox{$\widetilde{B}_1:\widetilde{\mathcal{D}}_{\Theta_{1}}\rightarrow K_{\Theta_{2}}$} and $\widetilde{B}_2: \widetilde{\mathcal{D}}_{\Theta_{2}}\rightarrow K_{\Theta_{1}}$, such that
	\begin{equation}\label{cara2}
	A-S_{\Theta_2}^*AS_{\Theta_1}=\widetilde{B}_1\dcf{1}+\dcf{2}\widetilde{B}_2^*.
	\end{equation}
\end{theorem}
\begin{proof}
	Let $A:K_{\Theta_{1}}\rightarrow  K_{\Theta_{2}}$ be a bounded linear operator. By Theorem \ref{ttau}, $A$ belongs to $\mathcal{MT}(\Theta_{1}, \Theta_{2})$  if and only if $ \widetilde{A} = \tau_{\Theta_2}A\, \tau_{\Theta_1}^*$ belongs to $\mathcal{MT}(\widetilde{\Theta}_{1}, \widetilde{\Theta}_{2})$. By Theorem \ref{tcara} the latter happens if and only if there exist bounded linear operators \mbox{$B_1:\mathcal{D}_{\widetilde{\Theta}_{1}}\to K_{\widetilde{\Theta}_{2}}$} and $B_2: \mathcal{D}_{\widetilde{\Theta}_{2}}\to K_{\widetilde{\Theta}_{1}}$, such that
	\begin{equation}\label{X}
	\widetilde{A}-S_{\widetilde{\Theta}_2}\widetilde{A}S_{\widetilde{\Theta}_1}^*=	\tau_{\Theta_2}A\, \tau_{\Theta_1}^*-S_{\widetilde{\Theta}_2}\tau_{\Theta_2}A\, \tau_{\Theta_1}^*S_{\widetilde{\Theta}_1}^*=B_1\dfct{1}+\dfct{2}B_2^*.
	\end{equation}
In other words,
\begin{equation*}
	A-\tau_{\Theta_2}^*S_{\widetilde{\Theta}_2}\tau_{\Theta_2}A\, \tau_{\Theta_1}^*S_{\widetilde{\Theta}_1}^*\tau_{\Theta_1}=\tau_{\Theta_2}^*B_1\dfct{1}\tau_{\Theta_1}+\tau_{\Theta_2}^*\dfct{2}B_2^*\,\tau_{\Theta_1}.
\end{equation*}
By \eqref{sz} we have
$$\tau_{\Theta_2}^*S_{\widetilde{\Theta}_2}\tau_{\Theta_2}=\tau_{\widetilde{\Theta}_2}S_{\widetilde{\Theta}_2}\tau_{\widetilde{\Theta}_2}^*=S_{\Theta_{2}}^*\ \text{  and  }\
\tau_{\Theta_1}^*S_{\widetilde{\Theta}_1}^*\tau_{\Theta_1}=\tau_{\widetilde{\Theta}_1}S_{\widetilde{\Theta}_1}\tau_{\widetilde{\Theta}_1}^*=S_{\Theta_{1}},$$
while from \eqref{ddd} it follows that
$$\dfct{1}\tau_{\Theta_1}=\tau_{\Theta_1}\dcf{1}\quad\text{and}\quad \tau_{\Theta_2}^*\dfct{2}=\dcf{2}\tau_{\Theta_2}^*.$$
Thus \eqref{X} is equivalent to
	\begin{equation*}
A-S_{\Theta_2}^*AS_{\Theta_1}=\tau_{\Theta_2}^*B_1\,\tau_{\Theta_1}\dcf{1}+\dcf{2}\tau_{\Theta_2}^*B_2^*\,\tau_{\Theta_1}=\widetilde{B}_1\dcf{1}+\dcf{2}\widetilde{B}_2^*.
\end{equation*}
with
$$\widetilde{B}_1=\tau_{\Theta_2}^*B_1\,\tau_{\Theta_1|\widetilde{\mathcal{D}}_{\Theta_{1}}},\quad \widetilde{B}_1:\widetilde{\mathcal{D}}_{\Theta_{1}}\to K_{\Theta_{2}}$$
and
$$\widetilde{B}_2=\big(\tau_{\Theta_2}^*B_2^*\,\tau_{\Theta_1}\big)
^*=\tau_{\Theta_1}^*B_2\,\tau_{\Theta_2|\widetilde{\mathcal{D}}_{\Theta_{2}}},\quad \widetilde{B}_2:\widetilde{\mathcal{D}}_{\Theta_{2}}\to K_{\Theta_{1}}.$$
Note that $\tau_{\Theta_i}^*{\mathcal{D}}_{\widetilde{\Theta_{i}}}=\widetilde{\mathcal{D}}_{\Theta_{i}}$, $i=1,2$. This allows us to treat $\tau_{\Theta_2}^*B_2^*\,\tau_{\Theta_1}$ as an operator from $K_{\Theta_{1}}$ to $\widetilde{\mathcal{D}}_{\Theta_{2}}$. Moreover, we have
\begin{equation}\label{wyk2}
B_1=\tau_{\Theta_2}\widetilde{B}_1
\tau_{\Theta_1|{\mathcal{D}}_{\widetilde{\Theta}_{1}}}^*\quad\text{and}\quad B_2=\tau_{\Theta_1}\widetilde{B}_2
\tau_{\Theta_2|{\mathcal{D}}_{\widetilde{\Theta}_{2}}}^*.
\end{equation}
\end{proof}
Note from the proof of Theorem \ref{tca2} that if $A:K_{\Theta_1}\to K_{\Theta_2}$ satisfies \eqref{cara2} with some $\widetilde{B}_1:\widetilde{\mathcal{D}}_{\Theta_1}\to K_{\Theta_2}$ and $\widetilde{B}_2:\widetilde{\mathcal{D}}_{\Theta_2}\to K_{\Theta_1}$, then
$\widetilde{A}=\tau_{\Theta_2}A\tau_{\Theta_1}^*$ satisfies \eqref{X} with $B_1$ and $B_2$ given by \eqref{wyk2}. By Corollary \ref{cep}, $\widetilde{A}=A_{\Psi+\Xi^*}^{\widetilde{\Theta}_1, \widetilde{\Theta}_2}$ with
$$\Psi(z)x=(B_1\ke_{0}^{\widetilde{\Theta}_1}x)(z)
=(\tau_{\Theta_2}\widetilde{B}_1\tau^*_{\Theta_1}\ke_{0}^{\widetilde{\Theta}_1}x)(z)=(\tau_{\Theta_2}\widetilde{B}_1\widetilde{\ke}_{0}^{\Theta_1}x)(z)$$
and
$$\Xi(z)x=(B_2\ke_{0}^{\widetilde{\Theta}_1}x)(z)=(\tau_{\Theta_1}\widetilde{B}_2\tau^*_{\Theta_2}\ke_{0}^{\widetilde{\Theta}_2}x)(z)
=(\tau_{\Theta_1}\widetilde{B}_2\widetilde{\ke}_{0}^{\Theta_2}x)(z).$$
Moreover (see Remark \ref{rmk}), $\Psi\in \mathcal{M}_{\widetilde{\Theta}_2}$ and $\Xi\in \mathcal{M}_{\widetilde{\Theta}_1}$.

It follows from Theorem \ref{ttau} (see \eqref{wykwyk}) that $A=A_{\Phi}^{\Theta_1, \Theta_2}$ with
\begin{align*}
\Phi(z)
&=\Theta_2(z)(\Psi(\overline{z})+\Xi(\overline{z})^*)\Theta_1(z)^*\\
&=\Theta_2(z)\Psi(\overline{z})\Theta_1(z)^*+\Theta_2(z)\Xi(\overline{z})^*\Theta_1(z)^*\\
&=\Theta_2(z)\widetilde{\Xi}(z)\Theta_1(z)^*+(\Theta_1(z)\widetilde{\Psi}(z)\Theta_2(z)^*)^*.
\end{align*}
By Lemma \ref{5.6} below, $\Phi=\Psi_1+\Xi_1$ with $\Psi_1=\Theta_2\widetilde{\Xi}\Theta_1^*\in
\Theta_2(z\mathcal{M}_{\Theta_1})^*$ and $\Xi_{1}=\Theta_1\widetilde{\Psi}\Theta_2^*\in \Theta_1(z\mathcal{M}_{\Theta_2})^*$.

\begin{lemma}\label{5.6}
Let $\Phi\in H^2(\oL)$. If $\Phi\in \mathcal{M}_\Theta$, then $\widetilde{\Phi}\widetilde{\Theta}^*\in (z\mathcal{M}_{\widetilde{\Theta}})^*$.
\end{lemma}
\begin{proof}
We will show that if $\Phi\in \mathcal{M}_\Theta$, then $\Psi(z)=\widetilde{\Theta}(z)\overline{z}\Phi(\overline{z})\in \mathcal{M}_{\widetilde{\Theta}}$.
Let $\hh\in H^2(\oL)$. Then
\begin{align*}
\langle \Psi, (z\hh)^*\rangle_{L^2(\oL)}
&=\int\limits_{\mathbb{T}}\langle \Psi, \overline{z}\hh(z)^*\rangle_2dm(z)
=\int\limits_{\mathbb{T}}\langle \widetilde{\Theta}(z)\overline{z}\Phi(\overline{z}), \overline{z}\hh(z)^*\rangle_2dm(z)\\
&=\int\limits_{\mathbb{T}}\langle \Theta(z)^*\Phi(z), \widetilde{\hh}(z)\rangle_2dm(z)
=\int\limits_{\mathbb{T}}\langle\Phi(z), \Theta(z)\widetilde{\hh}(z)\rangle_2dm(z)\\
&=\langle \Phi, z\widetilde{\hh}\rangle_{L^2(\oL)}=0,
\end{align*}
Moreover,
\begin{align*}
\langle \Psi, \widetilde{\Theta}\hh\rangle_{L^2(\oL)}
=\int\limits_{\mathbb{T}}\langle \widetilde{\Theta}(z)\overline{z}\Phi(\overline{z}), \widetilde{\Theta}\hh(z)\rangle_2dm(z)
&=\int\limits_{\mathbb{T}}\langle \widetilde{\Phi}(z)^*, z\hh(z)\rangle_2dm(z)\\
=\langle \widetilde{\Phi}^*, z\hh\rangle_{L^2(\oL)}=0,
\end{align*}
which means that $\Psi\in \mathcal{M}_{\widetilde{\Theta}}$.
\end{proof}
As in the scalar case, we can use Theorem \ref{tcara} and Theorem \ref{tca2} to get the following.
\begin{corollary}
Let $\Theta_1, \Theta_2\in H^\infty(\oL)$ be two pure inner functions and let $A:K_{\Theta_1}\to K_{\Theta_2}$ be a bounded linear operator. Then $A$ belongs to $\mathcal{MT}(\Theta_1, \Theta_2)$ if and only if the following hold:
\begin{enumerate}[(a)]
\item there exist bounded linear operators $\widehat{B}_1:\mathcal{D}_{\Theta_1}\to K_{\Theta_2}$ and $\widehat{B}_2:\widetilde{{\mathcal{D}}}_{\Theta_2}\to K_{\Theta_1}$, such that
    $$S_{\Theta_2}^*A-AS_{\Theta_1}^*=\widehat{B}_1D_{\Theta_1}+\widetilde{D}_{\Theta_2}\widehat{B}_2^*.$$\label{i}
    \item there exist bounded linear operators $\widehat{B}_1:\widetilde{\mathcal{D}}_{\Theta_1}\to K_{\Theta_2}$ and $\widehat{B}_1:\mathcal{D}_{\Theta_2}\to K_{\Theta_1}$, such that
        $$S_{\Theta_2}A-AS_{\Theta_1}=\widehat{B}_1\widetilde{D}_{\Theta_1}+D_{\Theta_2}\widehat{B}_{2}^*.$$\label{ii}
\end{enumerate}
\end{corollary}
\begin{proof}
The proof is similar to the scalar case (see \cite{BM}). To prove \eqref{i} assume first that $A\in \mathcal{MT}(\Theta_1, \Theta_2)$. Then, by Theorem \ref{tcara}, there exist bounded linear operators $B_1:\mathcal{D}_{\Theta_1}\to K_{\Theta_2}$ and $B_2:\mathcal{D}_{\Theta_2}\to K_{\Theta_1}$, such that $$A-S_{\Theta_2}AS_{\Theta_1}^*=B_1D_{\Theta_1}+D_{\Theta_2}B_2^*.$$
Hence $$S^*_{\Theta_2}A-S^*_{\Theta_2}S_{\Theta_2}AS_{\Theta_1}^*=S^*_{\Theta_2}B_1D_{\Theta_1}+S^*_{\Theta_2}D_{\Theta_2}B_2^*,$$
and since $S_{\Theta}^*S_{\Theta_2}=I_{K_{\Theta_2}}-\widetilde{D}_{\Theta_2}$, we get
$$A-S_{\Theta_2}AS_{\Theta_1}^*=S^*_{\Theta_2}B_1D_{\Theta_1}+S^*_{\Theta_2}D_{\Theta_2}B_2^*-D_{\Theta_2}AS^*_{\Theta_1}.$$
Observe now that $S^*_{\Theta_2}D_{\Theta_2}=\widetilde{D}_{\Theta_2}S^*_{\Theta_2}$ and $\widetilde{D}_{\Theta_2}=\widetilde{D}_{\Theta_2}P_{\widetilde{D}_{\Theta_2}}$, where $P_{\widetilde{D}_{\Theta_2}}$ is the orthogonal projection from $K_{\Theta_2}$ to $\widetilde{D}_{\Theta_2}$ (see the formula for $\widetilde{D}_{\Theta_2}$ on page 11). It follows that
\begin{align*}
A-S_{\Theta_2}AS_{\Theta_1}^*
&=S^*_{\Theta_2}B_1D_{\Theta_1}+\widetilde{D}_{\Theta_2}(S_{\Theta_2}^*B_2^*-AS_{\Theta_1}^*)\\
&=\widehat{B}_1D_{\Theta_1}+\widehat{D}_{\Theta_2}\widehat{B}^*_{2},
\end{align*}
where $$\widehat{B}_1=S^*_{\Theta_2}B_1: \mathcal{D}_{\Theta_2}\to K_{\Theta_1}.$$
and
$$\widehat{B}_2=(P_{\widetilde{D}{\Theta_2}}(S^*_{\Theta_2}B^*_2-AS^*_{\Theta_1}))^*:\widetilde{\mathcal{D}}_{\Theta_2}\to K_{\Theta_1}.$$
The proof of the other implication is analogous.

To prove \eqref{ii} one can apply the same reasoning together with Theorem \ref{tca2}. Alternatively, one can use the fact that $A\in \mathcal{MT}(\Theta_1, \Theta_2)$ if and only if $\tau_{\Theta_2}A\tau^*_{\Theta_1}\in \mathcal{MT}(\widetilde{\Theta}_1, \widetilde{\Theta}_2)$ to show that \eqref{ii} is equivalent to \eqref{i}.
\end{proof}
\section{Shift invariance and MATTO's}
Shift invariance for TTO's was introduced in \cite{sar}. D. Sarason proved that a bounded linear operator $A:K_\theta\to K_\theta$ is a TTO if and only if it is shift invariant, i.e.,
$$\langle ASf, Sf\rangle_{L^2}=\langle Af, f\rangle_{L^2} \quad \text{for each} ~~f\in K_\theta~~ \text{such that }~~ Sf\in K_\theta.$$
In \cite{KT} we prove that the same is true for MTTO's.

 Here we consider shift invariance of MATTO's. As in the scalar case (see \cite{BM}), we characterize MATTO's in term of four (equivalent) types of shift invariance.

Recall that for an operator valued inner function $\Theta\in H^\infty(\oL)$ and for $f\in K_{\Theta}$ we have
$$Sf=M_zf\in K_{\Theta}~~\text{if and only if} ~~ f\perp \widetilde{\mathcal{D}}_{\Theta}~~~ (\tau_\Theta f(0)=0)$$
and
$$S^*f=M_{\overline{z}}f\in K_{\Theta} ~~ \text{if and only if} ~~ f\perp\mathcal{D}_{\Theta}~~ ~(f(0)=0).$$
\begin{theorem}
Let $\Theta_1, \Theta_2\in H^\infty(\oL)$ be two pure inner functions and let $A:K_{\Theta_1}\to K_{\Theta_2}$ be a bounded linear operator. Then $A$ belongs to $\mathcal{MT}(\Theta_1, \Theta_2)$ if and only if it has one (and all) of the following properties:
\begin{enumerate}[(a)]
\item $\langle AS^*f, S^*g\rangle_{L^2(\hsp)}=\langle A f,g\rangle_{L^2(\hsp)}$ for all $f\in K_{\Theta_1}$, $g\in K_{\Theta_2}$ such that $f\perp\mathcal{D}_{\Theta_1}$, $g\perp\mathcal{D}_{\Theta_2}$;\label{T1}
    \item $\langle AS^*f, g\rangle_{L^2(\hsp)}=\langle Af,Sg\rangle_{L^2(\hsp)}$ for all $f\in K_{\Theta_1}$, $g\in K_{\Theta_2}$ such that $f\perp\mathcal{D}_{\Theta_1}$, $g\perp\widetilde{\mathcal{D}}_{\Theta_2}$;\label{T2}
    \item $\langle ASf, Sg\rangle_{L^2(\hsp)}=\langle Af,g\rangle_{L^2(\hsp)}$ for all $f\in K_{\Theta_1}$, $g\in K_{\Theta_2}$ such that $f\perp\widetilde{\mathcal{D}}_{\Theta_1}$, $g\perp\widetilde{\mathcal{D}}_{\Theta_2}$;\label{T3}
    \item $\langle ASf, g\rangle_{L^2(\hsp)}=\langle Af,S^*g\rangle_{L^2(\hsp)}$ for all $f\in K_{\Theta_1}$, $g\in K_{\Theta_2}$ such that $f\perp\widetilde{\mathcal{D}}_{\Theta_1}$, $g\perp\mathcal{D}_{\Theta_2}$;\label{T4}
\end{enumerate}
\end{theorem}
\begin{proof}
\eqref{T1} If $A\in \mathcal{MT}(\Theta_1, \Theta_2)$, then by Theorem \ref{tcara},
$$A-S_{\Theta_2}AS_{\Theta_1}^*=B_1D_{\Theta_1}+D_{\Theta_2}B_2^*$$
for some bounded linear operators $B_1: \mathcal{D}_{\Theta_1}\to K_{\Theta_2}$ and $B_2: \mathcal{D}_{\Theta_2}\to K_{\Theta_1}$. It follows that for all $f\in K_{\Theta_1}$, $f\in K_{\Theta_2}$ such that $f\perp\mathcal{D}_{\Theta_1}$, $g\perp\mathcal{D}_{\Theta_2}$, we have
\begin{align*}
\langle AS^*f, S^*g\rangle_{L^2(\hsp)}
&=\langle A S^*_{\Theta_1}f,S^*_{\Theta_2}g\rangle_{L^2(\hsp)}=\langle  S_{\Theta_2}AS^*_{\Theta_2}f, g\rangle_{L^2(\hsp)}\\
&=\langle Af, g\rangle_{L^2(\hsp)}-\langle B_1D_{\Theta_1}f,g\rangle_{L^2(\hsp)}-\langle D_{\Theta_2}B_2^*f,g\rangle_{L^2(\hsp)}.
\end{align*}
Since $D_{\Theta_1}f=0$ and $D_{\Theta_2}B_2^*f\in \mathcal{D}_{\Theta_2}$, we get
\begin{equation}\label{Sin1}
\langle AS^*f, S^*g\rangle_{L^2(\hsp)}=\langle A f,g\rangle_{L^2(\hsp)}
\end{equation}
On the other hand, if \eqref{Sin1} holds for all $f\in K_{\Theta_1}$, $g\in K_{\Theta_2}$ such that $f\perp \mathcal{D}_{\Theta_1}$, $g\perp \mathcal{D}_{\Theta_2}$, we have
$$\langle (A-S_{\Theta_2}AS_{\Theta_1}^*)f,g\rangle_{L^2(\hsp)}=\langle Af,g\rangle_{L^2(\hsp)}- \langle AS^*f,S^*g\rangle_{L^2(\hsp)}=0.$$
This means that the operator $\T_A=A-S_{\Theta_2}AS_{\Theta_1}^*$ maps $\mathcal{D}^{\perp}_{\Theta_1}$ into $\mathcal{D}_{\Theta_2}$, or in other words,
\begin{equation}\label{Sin2}
(I_{K_{\Theta_2}}-P_{\mathcal{D}_{\Theta_2}})\T_A(I_{K_{\Theta_1}}-P_{\mathcal{D}_{\Theta_1}})=0,
\end{equation}
where $P_{\mathcal{D}_{\Theta_i}}$ is the orthogonal projection from $K_{\Theta_i}$ onto $\mathcal{D}_{\Theta_i}$, $i=1,2$.
Recall now that
$$\text{Range}P_{\mathcal{D}_{\Theta_i}}=\mathcal{D}_{\Theta_i}=\text{Range}D_{\Theta_i},~~~ i=1,2,$$
and so there exist bounded linear operators $R_i: K_{\Theta_i}\to K_{\Theta_i}$, $i=1,2,$ such that
$$P_{\mathcal{D}_{\Theta_i}}=D_{\Theta_i}R_i=R_i^*D_{\Theta_i}, ~~~ i=1,2$$
(the second equality follows from the fact that $P_{\mathcal{D}_{\Theta_i}}^*=P_{\mathcal{D}_{\Theta_i}}$). Together with \eqref{Sin2} this gives
\begin{align*}
A-S_{\Theta_2}AS_{\Theta_1}^*
&=\T_A=P_{\mathcal{D}_{\Theta_2}}\T_A+\T_AP_{\mathcal{D}_{\Theta_2}}-P_{\mathcal{D}_{\Theta_2}}\T_AP_{\mathcal{D}_{\Theta_1}}\\
&=D_{\Theta_2}R_2\T_A+(I_{K_{\Theta_2}}-P_{\mathcal{D}_{\Theta_2}})\T_AR_1^*D_{\Theta_1}
\end{align*}
and so $A$ satisfies \eqref{cara} with
$$B_1=(I_{K_{\Theta_2}}-P_{\mathcal{D}_{\Theta_2}})
\T_AR^*_{1|\mathcal{D}_{\Theta_1}}:
\mathcal{D}_{\Theta_1}\to K_{\Theta_2}$$
and
$$B_2=(P_{\mathcal{D}_{\Theta_2}}R_2\T_A)^*=\T_A^*R^*_{2|\mathcal{D}_{\Theta_2}}:\mathcal{D}_{\Theta_2}\to K_{\Theta_1}.$$
By Theorem \ref{tcara}, $A\in \mathcal{MT}(\Theta_1, \Theta_2)$.

%

The proof of \eqref {T2},\eqref{T3} and \eqref{T4} is analogous to the proof of \eqref{T1}.
\end{proof}

\section{Characterization with modified compressed shift operators}
  Modified compressed shifts were introduced by Sarason in \cite[section 10]{sar}. For any nonconstant inner function $\Theta$, suppose that $X_\Theta : \widetilde{\mathcal{D}}_{\Theta}\rightarrow \mathcal{D}_{\Theta}$, and consider $\widehat{X}_\Theta\in \mathcal{ L}(K_\Theta)$ defined by $\widehat{X}_\Theta f=X_\Theta P_{\widetilde{\mathcal{D}}_{\Theta}}f$.
The operator modified shift is  defined by
$$S_{\Theta, X_\Theta}=S_{\Theta}+(\widehat{X}_\Theta-S_{\Theta})P_{\widetilde{\mathcal{D}}_{\Theta}}, $$
or $$S_{\Theta, X_\Theta}=S_{\Theta}+P_{\mathcal{D}_{\Theta}}Y_{\Theta}P_{\widetilde{\mathcal{D}}_{\Theta}}, $$
which implies that $$S_{\Theta}=S_{\Theta, X_\Theta}-P_{\mathcal{D}_{\Theta}}Y_{\Theta}P_{\widetilde{\mathcal{D}}_{\Theta}}$$
where $Y_{\Theta}=\widehat{X}_{\Theta}-S_{\Theta}$.
\begin{theorem} Let $ {\Theta_1}, {\Theta_2}\in H^\infty (\oL)$ be two pure inner functions. Let $A:K_{\Theta_{1}}\rightarrow  K_{\Theta_{2}}$ be a bounded operator.
Then $A\in\mathcal{MT}(\Theta_{1}, \Theta_{2})$  if and only if
\begin{equation}
A-S_{\Theta_2, X_{\Theta_2}}AS_{\Theta_1, X_{\Theta_1}}^*=B P_{\mathcal{D}_{\Theta_1}}+P_{\mathcal{D}_{\Theta_2}}B^{'*}.
\end{equation}
\end{theorem}
\begin{proof}
Consider
\begin{eqnarray*}
 A-S_{\Theta_2}AS_{\Theta_1}^*& =& 	 A-(S_{\Theta_2, X_{\Theta_{2}}}-P_{\mathcal{D}_{\Theta_2}}Y_{\Theta_2}P_{\widetilde{\mathcal{D}}_{\Theta_2}})A
 (S_{\Theta_1,X_{\Theta_{1}}}^*-P_{\widetilde{\mathcal{D}}_{\Theta_1}}Y_{\Theta_1}^*P_{\mathcal{D}_{\Theta_1}})\\
 & =& 	A-S_{\Theta_2,X_{\Theta_{2}}}AS_{\Theta_1,X_{\Theta_{1}}}^*+S_{\Theta_{2},X_{\Theta_{2}}}
 P_{\widetilde{\mathcal{D}}_{\Theta_{1}}}Y_{\Theta_{1}}^{*}P_{\mathcal{D}_{\Theta_{1}}}\\
 &+&
 P_{\mathcal{D}_{\Theta_{2}}}Y_{\Theta_{2}}P_{\widetilde{\mathcal{D}}_{\Theta_{2}}}AS_{\Theta_{1},X_{\Theta_{1}}}^{*}
 -P_{\mathcal{D}_{\Theta_{2}}}Y_{\Theta_{2}}P_{\widetilde{\mathcal{D}}_{\Theta_{2}}}A
 P_{\widetilde{\mathcal{D}}_{\Theta_{1}}}Y_{\Theta_{1}}^{*}P_{\mathcal{D}_{\Theta_{1}}}\\
 & =& 	A-S_{\Theta_2,X_{\Theta_{2}}}AS_{\Theta_1,X_{\Theta_{1}}}^*+S_{\Theta_{2},X_{\Theta_{2}}}
 P_{\widetilde{\mathcal{D}}_{\Theta_{1}}}Y_{\Theta_{1}}^{*}P_{\mathcal{D}_{\Theta_{1}}}\\
 &+&P_{\mathcal{D}_{\Theta_{2}}}[Y_{\Theta_{1}}^{*}P_{\widetilde{\mathcal{D}}_{\Theta_{2}}}AS_{\Theta_{1},X_{\Theta_{1}}}^{*}
 -Y_{\Theta_{2}}^{*}P_{\widetilde{\mathcal{D}}_{\Theta_{2}}}A
 P_{\widetilde{\mathcal{D}}_{\Theta_{1}}}Y_{\Theta_{1}}^{*}P_{\mathcal{D}_{\Theta_{1}}}]\\
 & =& BP_{\mathcal{D}_{\Theta_{1}}}+P_{\mathcal{D}_{\Theta_{2}}}B^{\prime *}+T_{1}P_{\mathcal{D}_{\Theta_{1}}}+P_{\mathcal{D}_{\Theta_{2}}}T_{2}\\
 &=&(B+T_{1})P_{\mathcal{D}_{\Theta_{1}}}+P_{\mathcal{D}_{\Theta_{2}}}(B^{\prime*}+T_{2}),
 \end{eqnarray*}	
where $T_1=S_{\Theta_{2},X_{\Theta_{2}}}
 P_{\widetilde{{}\mathcal{D}}_{\Theta_{1}}}Y_{\Theta_{1}}^{*}$ and $T_2=Y_{\Theta_{1}}^{*}P_{\widetilde{\mathcal{D}}_{\Theta_{2}}}AS_{\Theta_{1},X_{\Theta_{1}}}^{*}
 -Y_{\Theta_{2}}^{*}P_{\widetilde{\mathcal{D}}_{\Theta_{2}}}AP_{\widetilde{\mathcal{D}}_{\Theta_{1}}}Y_{\Theta_{1}}^{*}P_{\mathcal{D}_{\Theta_{1}}}.$ From equation (3.8) of \cite{KT} it follows that there is an operator $J_{\Theta_{1}}\in \mathcal{L}(K_{\Theta_{1}})$ such that  $$P_{\mathcal{D}_{\Theta_{1}}}=(I-S_{\Theta_{1}}S_{\Theta_{1}}^*)J_{\Theta_{1}}=D_{\Theta_{1}}J_{\Theta_{1}}=J_{\Theta_{1}}^*D_{\Theta_{1}},$$
  and similarly there is $J_{\Theta_{2}}\in \mathcal{L}(K_{\Theta_{2}})$ such that
  $$P_{\mathcal{D}_{\Theta_{2}}}=(I-S_{\Theta_{2}}S_{\Theta_{2}}^*)J_{\Theta_{2}}=D_{\Theta_{2}}J_{\Theta_{2}}=J_{\Theta_{2}}^*D_{\Theta_{2}}.$$
  Then we have
\begin{eqnarray*}
  A-S_{\Theta_2}AS_{\Theta_1}^*
  & =&(B+T_{1})J_{\Theta_{1}}^{*}D_{\Theta_{1}}+D_{\Theta_{2}}J_{\Theta_{2}}(B^{\prime *}+T_{2})\\
   & =&(B+T_{1})J_{\Theta_{1}}^{*}D_{\Theta_{1}}+D_{\Theta_{2}}[(B^{\prime}+T_{2}^{*})J_{\Theta_{2}}^{*}]^{*}\\
   &=& \B D_{\Theta_{1}}+D_{\Theta_{2}}\B^{\prime *}
 \end{eqnarray*}
 where $\B = (B+T_{1})J_{\Theta_{1}}^{*}$ and $\B^{\prime} = (B^{\prime}+T_{2}^{*})J_{\Theta_{2}}^{*}$.  The required result follows from this and Theorem \ref{tcara}.
\end{proof}

\medskip

\noindent$\bf{Acknowledgment}$

\noindent The first author is supported by the project TUBITAK 1001, 123F356.


\begin{thebibliography}{99}
	
	\bibitem{BCT}
	{A. Baranov, I. Chalendar, E. Fricain, J. Mashreghi and D. Timotin},
	\textit{Bounded symbols and reproducing kernel thesis for truncated Toeplitz operators},
	J. Funct. Anal. {\bf 259}, 2010, 2673--2701.
	
	
	\bibitem{berc}
{H. Bercovici},
\textit{Operator theory and aritmetic in $H^\infty$,} Mathematical surveys and monographs No. 26, Amer. Math. Soc., Procidence, Rhode Island 1988.
	
	
	
	
	\bibitem{BCKP}
	{C. C\^{a}mara, J. Jurasik, K. Kli\'s-Garlicka and M. Ptak}, \textit{Characterizations of asymmetric truncated Toeplitz operators,} Banach J. Math. Anal. {\bf 11} (2017), 899--922.
	
	
	\bibitem{CKP}
	{C. C\^{a}mara, K. Kli\'s-Garlicka and M. Ptak},
	\textit{Asymmetric truncated Toeplitz operators and conjugations},
	Filomat, {\bf 33} (2019), 3697--3710.

\bibitem{CKLP}
	M. C. C\^amara, K. Kli\'{s}-Garlicka, B. \L anucha, M. Ptak,  \emph{Conjugations in $L^2(\mathcal {H})$}, Integr. Equ. Oper. Theory \textbf{92} (2020), 48.

\bibitem{GP}
{S. R. Garcia and M. Putinar}, \textit{Complex symmetric operators and applications}, Trans. Amer. Math. Soc., {\bf 358} (2006), 1285--1315.


\bibitem{BM}
C. Gu, B. \L anucha, M. Michalska, \textit{Characterizations of asymmetric
	truncated Toeplitz and Hankel operators,} Complex Anal. Oper. Theory, 13, 673--684 (2019).	

\bibitem{blicharz1}
J. Jurasik, B. \L anucha, \textit{Asymmetric truncated Toeplitz operators equal to the zero operator,} Ann. Univ. Mariae Curie-Sk\l odowska, Sect. A \textbf{70} (2016), no. 2, 51--62.

\bibitem{jl}
J. Jurasik, B. \L anucha, \textit{Asymmetric truncated Toeplitz operators on finite-dimensional spaces,} Operators and Matrices \textbf{11} (2017), no. 1, 245--262.

\bibitem{RK}
R. Khan, \textit{The generalized Crofoot transform}, Oper. Matrices {\bf15} (1), 225-237 (2021).


\bibitem{RK2}
R. Khan, A. Farooq  \textit{ Generalized Crofoot transform and applications}, Conc. Operators {\bf10}, 2022138 (2023).

\bibitem{KT}
{R. Khan and D. Timotin}, \textit{Matrix valued trancated Toeplitz operators: Basic properties},
Complex Anal. Oper. Theory {\bf 12} (2018), 997--1014.
%


\bibitem{NF}
{B. Sz.-Nagy, C. F. Foias, H. Bercovici, L. K\'{e}rchy},
\textit{Harmonic analysis of operators on a Hilbert space}, second edition,  Springer, London 2010.


\bibitem{sar}
{D. Sarason}
\textit{Algebraic properties of truncated Toeplitz operators}, Oper. Matrices {\bf 1} (2007), 491--526.



\end{thebibliography}
\end{document}